\documentclass[11pt]{article}
\usepackage{amsthm, amsmath, amssymb, amsfonts, url, booktabs, tikz, setspace, fancyhdr, bm, mathrsfs}
\usepackage{hyperref}
\usepackage{geometry}
\usepackage{hyperref, enumerate}
\usepackage[shortlabels]{enumitem}
\usepackage[babel]{microtype}
\usepackage[english]{babel}
\usepackage[capitalise]{cleveref}
\usepackage{comment}
\usepackage{bbm}
\usepackage{csquotes}
\usepackage{mathabx}
\usepackage{tikz}
\usepackage{graphicx}
\usepackage{float}
\usepackage{amsmath}

\counterwithin{figure}{section}

\newtheorem{theorem}{Theorem}[section]
\newtheorem{prop}[theorem]{Proposition}

\newtheorem{conj}[theorem]{Conjecture}
\newtheorem{claim}[theorem]{Claim}

\newtheorem{fact}[theorem]{Fact}

\theoremstyle{definition}

\newtheorem*{defn-non}{Definition}

\newtheorem{ques}[theorem]{Question}
\usepackage[linesnumbered, ruled]{algorithm2e}
\SetKwRepeat{Do}{do}{while}%

\newenvironment{poc}{\begin{proof}[Proof of the claim]}{\end{proof}}

\usepackage{todonotes}

\newcommand{\cA}{\mathcal{A}}

\newcommand{\cB}{\mathcal{B}}

\newcommand{\cF}{\mathcal{F}}
\newcommand{\cG}{\mathcal{G}}
\newcommand{\cH}{\mathcal{H}}

\newcommand{\VC}{\operatorname{VC}}
\newcommand{\Tr}{\operatorname{Tr}}
\newcommand{\AK}{\operatorname{AK}}

\title{Beating the Ahlswede--Khachatrian bound for the Erd\H{o}s--Frankl--Pach problem}
\author{
Tuan Tran\thanks{School of Mathematical Sciences, University of Science and Technology of China, Hefei, China. Supported by the Excellent
Young Talents Program (Overseas) of the National Natural Science Foundation of China under Grant No.
GG0010007003. Email: trantuan@ustc.edu.cn.}
\and
Zixiang Xu\thanks{School of Mathematical Sciences, Zhejiang University, Hangzhou, China. Email: zixiangxu@zju.edu.cn.}
}
\date{}

\begin{document}
\maketitle

\begin{abstract}
In the 1980s, Erd\H{o}s and, independently, Frankl and Pach conjectured that, for sufficiently large \(n\), every \((d+1)\)-uniform family on \(\{1,\ldots,n\}\) with VC-dimension \(d\) has size at most \(\binom{n-1}{d}\), the size of a star. 
Ahlswede and Khachatrian disproved this conjecture in 1997 by giving a family of size $\binom{n-1}{d}+\binom{n-4}{d-2}$.
This value has since been widely believed to be best possible, and Mubayi and Zhao explicitly conjectured its optimality in 2007.
Very recently, Wang, Xu and Zhang proved their conjecture for
\(d=2\) and \(n\ge 7\), providing further support for this belief.

Surprisingly, we show that the Mubayi-Zhao conjecture is false for every \(d\ge 3\) by constructing families larger than the Ahlswede--Khachatrian bound.
Our constructions suggest that the answer to the Erd\H{o}s--Frankl--Pach problem depends delicately on both \(n\) and \(d\).
\end{abstract}

\section{Introduction}

Let \(\cF\subseteq 2^X\) be a set system.  A set \(S\subseteq X\) is
\emph{shattered} by \(\cF\) if \(\{F\cap S:F\in\cF\}=2^S\). The
\emph{VC-dimension} of \(\cF\), denoted by \(\VC(\cF)\), is the largest
size of a 
set shattered by \(\cF\).  The classical Sauer--Shelah
theorem~\cite{1972JCTASauer,1972PACJMShelah,1971TPAVCVC} states that if
\(\cF\subseteq 2^{[n]}\) and \(\VC(\cF)\le d\), then
\(|\cF|\le \sum_{i=0}^d\binom{n}{i}\), with equality attained by the Hamming
ball of radius \(d\).  

Over 40 years ago, Erd\H{o}s \cite{1984Erdos} and, independently, Frankl and Pach \cite{1984Franklpach} initiated the study of the uniform analogue of the Sauer--Shelah theorem. The problem has long attracted interest \cite{1997CombFan,2007JAC}, but substantial progress has come only in recent years \cite{2025CombProof,2024FranklPach,2025ThreeUniform,2025YangYu}.
For \(d\ge 0\), let \(M_d(n)\) denote the maximum size of a family
\(\cF\subseteq\binom{[n]}{d+1}\) with \(\VC(\cF)\le d\).
It is not hard to see that \(M_0(n)=1\) and \(M_1(n)=n-1\). For $d\ge 2$, however, much less is known.


\begin{ques}[Erd\H{o}s--Frankl--Pach]\label{question:EFP}
For given integers \(d\ge 2\) and \(n\ge 2d+2\), determine \(M_d(n)\).
\end{ques}

This problem was originally motivated by a possible
generalization of the Erd\H{o}s--Ko--Rado theorem.  Indeed, the canonical star
\(\{F\in\binom{[n]}{d+1}:1\in F\}\) has size \(\binom{n-1}{d}\) and
VC-dimension at most \(d\).  Erd\H{o}s~\cite{1984Erdos}, as well  Frankl and
Pach~\cite{1984Franklpach}, conjectured that this 
construction is optimal whenever
\(n\) is sufficiently large compared to \(d\).  Frankl and Pach~\cite{1984Franklpach} also proved
the elegant upper bound \(M_d(n)\le \binom{n}{d}\), 
which remains a
fundamental benchmark. Later improvements include the \(\Omega_d(\log n)\)
saving of Mubayi and Zhao~\cite{2007JAC} for prime power \(d\) and sufficiently large \(n\), 
and the result of Ge, Xu, Yip, Zhang and Zhao~\cite{2024FranklPach} that the Frankl--Pach bound is never attained. 
More recently, Chao, Xu, Yip and Zhang~\cite{2025CombProof} gave a purely combinatorial proof that \(M_d(n)\le \binom{n-1}{d} +O_d(n^{d-1-\frac1{4d-2}})\), which Yang and Yu~\cite{2025YangYu} subsequently sharpened to \(M_d(n)\le \binom{n-1}{d}+O_d(n^{d-2})\).

On the lower bound side, the star conjecture was disproved by Ahlswede and
Khachatrian~\cite{1997CombFan}, who constructed a \((d+1)\)-uniform family
with VC-dimension at most \(d\) and size
\(\AK_d(n):=\binom{n-1}{d}+\binom{n-4}{d-2}\).  
Mubayi and Zhao~\cite{2007JAC} later found infinitely many non-isomorphic constructions of the same size and conjectured that $M_d(n)=\AK_d(n)$ for all sufficiently large \(n\). 
For convenience, we recall their construction. Let
\(V=[n]\setminus\{1,2\}\), and define
\[
        \cG_1=
        \binom{V}{d}\setminus
        \bigg\{G\in\binom{V}{d}:4\in G,\ 3\notin G\bigg\}\ \textup{and}\  \cG_2=\left\{G\in\binom{V}{d}:4\in G\right\}.
\]
Set
\[
        \cA=
        \bigg\{F\in\binom{[n]}{d+1}:\{1,2\}\subseteq F\bigg\}
        \cup\big\{\{1\}\cup G:G\in\cG_1\big\}
        \cup\big\{\{2\}\cup G:G\in\cG_2\big\}.
\]
Ahlswede and Khachatrian~\cite{1997CombFan} proved that \(\VC(\cA)\le d\), and 
a direct calculation gives
\[
        |\cA|
        =\binom{n-1}{d}+\binom{n-4}{d-2}=\AK_d(n).
\]
The case \(d=2\) provides strong evidence 
for this 
conjecture: Wang, Xu and Zhang~\cite{2025ThreeUniform} proved that
\(M_2(n)=\AK_2(n)\) for all \(n\ge 7\).  
Thus, for 3-uniform families with VC-dimension at most \(2\), the Ahlswede--Khachatrian bound is optimal. 


Our main result disproves the Mubayi--Zhao conjecture for every \(d\ge 3\).

\begin{theorem}\label{thm:main}
The following statements hold.
\begin{itemize}
\item[\rm (i)] For $n\ge 6$, 
\[
M_3(n)\ge \binom{n-1}{3}+n-3.
\]
\item[\rm (ii)] For every $d\ge 4$ and $n\ge 2d$,
\[
        M_d(n)\ge \binom{n-1}{d}+\binom{n-4}{d-2}+\binom{n-6}{d-3}+2M_{d-4}(n-6).
\]
\end{itemize}
\end{theorem}

Thus, \cref{thm:main} improves the Ahlswede--Khachatrian bound by
by at least
\(\binom{n-6}{d-3}\).  
The exact result for \(d=2\) can also be used as a seed in the recursive term. For
example, for \(n\ge 13\), we have \(M_6(n)\ge \AK_6(n)+\binom{n-6}{3}+2\binom{n-7}{2}+2\).

We find the phenomenon behind this improvement quite striking.
For \(d=2\), the Ahlswede--Khachatrian construction is sharp because it leaves no lower dimensional room to fill. 
In higher dimensions, however, missing patterns can be filled
recursively. 
This suggests that the Erd\H{o}s--Frankl--Pach problem may not 
admit a uniform answer across all dimensions: the correct formula may genuinely
change with \(d\).

\section{Proof of Theorem~\ref{thm:main}(ii)}
In this section, we prove \cref{thm:main}(ii). 
The same core construction, without the recursive part, yields the bound $M_3(n)\ge \binom{n-1}{3}+n-3$ in Theorem~\ref{thm:main}(i); we give the details in Appendix~\ref{appendix}. 
Throughout, if
\(P=\{i_1,\ldots,i_t\}\), we write \(i_1\cdots i_t\) for \(P\).  For a family
\(\cF\) and a set \(S\), write
\[
        \Tr_{\cF}(S)=\{F\cap S:F\in\cF\}.
\]

We shall use the following elementary criterion.

\begin{fact}\label{lem:missing-trace}
Let \(\cF\subseteq \binom{[n]}{d+1}\).  Suppose that for every
\(F\in\cF\), there is a set \(T_F\subseteq F\) such that
\(
        T_F\notin \Tr_{\cF}(F).
\)
Then \(\VC(\cF)\le d\).
\end{fact}
Let \(C=[6]=\{1,2,3,4,5,6\}\) and \(U=[n]\setminus C\).
Define \(\cB=\bigsqcup_{i=1}^{6}\mathcal{B}_{i}\subseteq 2^C\) by
\[
\begin{aligned}
\cB_1={}&\{1\},\\
\cB_2={}&\{12,15,23,24,26\},\\
\cB_3={}&\{123,124,125,126,135,156,234,236,246,256,456\},\\
\cB_4={}&\{1234,1235,1236,1246,1256,1345,1346,1356,1456,2345,2346,2456,3456\},\\
\cB_5={}&\{12345,12346,12356,12456,13456,23456\},\\
\cB_6={}&\{123456\}.
\end{aligned}
\]
The level sizes are
\(
        (|\cB_1|,\ldots,|\cB_6|)
        =(1,5,11,13,6,1).
\)
Let the two missing \(4\)-patterns be \(R_1=1245\) and \(R_2=2356\).  Both
\(R_1\) and \(R_2\) are absent from \(\cB\).  Choose two families
\(\cH_1,\cH_2\subseteq\binom{U}{d-3}\) with
\(|\cH_1|=|\cH_2|=M_{d-4}(n-6)\) and \(\VC(\cH_i)\le d-4\) for
\(i=1,2\). Note that when $d=4$, we have \(\cH_1=\{x_1\}\) and \(\cH_2=\{x_2\}\) for some \(x_1,x_2\in U\).
Define
\begin{equation*}
\cF_6=\bigg\{P\cup B:P\in\cB,\ B\in\binom{U}{d+1-|P|}\bigg\}\cup \{R_1\cup H:H\in\cH_1\}\cup \{R_2\cup H:H\in\cH_2\}.
\end{equation*}
We call these three subfamilies the \(\cB\)-part, the
\(R_1\)-recursive part, and the \(R_2\)-recursive part, respectively.
The three parts are pairwise disjoint, since the three possible core patterns
in \(C\) are different and \(R_1,R_2\notin\cB\).  Hence, with \(N=n-6\),
\[
\begin{aligned}
|\cF_6|
&=\sum_{s=1}^6 |\cB_s|\binom{N}{d+1-s}+2M_{d-4}(n-6)\\
&=\binom{N}{d}
+5\binom{N}{d-1}
+11\binom{N}{d-2}
+13\binom{N}{d-3}
+6\binom{N}{d-4}
+\binom{N}{d-5}
+2M_{d-4}(n-6)\\
&=\left(\sum_{j=0}^5\binom{5}{j}\binom{N}{d-j}\right)
+\left(\binom{N}{d-2}+2\binom{N}{d-3}+\binom{N}{d-4}\right)
+\binom{N}{d-3}
+2M_{d-4}(n-6)\\
&=\binom{N+5}{d}+\binom{N+2}{d-2}+\binom{N}{d-3}
+2M_{d-4}(n-6)\\
&=\binom{n-1}{d}+\binom{n-4}{d-2}+\binom{n-6}{d-3}
+2M_{d-4}(n-6).
\end{aligned}
\]
It then suffices to show the following property on \(\cF_6.\)
\begin{prop}\label{prop:VCD}
   \(\VC(\cF_6)\le d\). 
\end{prop} 
\begin{proof}[Proof of Proposition~\ref{prop:VCD}]
For every
\(P\in\cB\setminus\{1345,3456\}\), define
\(\sigma(P)\subsetneq P\) by the following table:
\[
\begin{array}{c|c@{\hspace{1.5em}}c|c@{\hspace{1.5em}}c|c}
P&\sigma(P)&P&\sigma(P)&P&\sigma(P)\\
\hline
1&\emptyset & 12&\emptyset & 15&5\\
23&3 & 24&4 & 26&6\\
123&3 & 124&\emptyset & 125&\emptyset\\
126&\emptyset & 135&35 & 156&5\\
234&34 & 236&36 & 246&4\\
256&6 & 456&45 & 1234&\emptyset\\
1235&\emptyset & 1236&\emptyset & 1246&\emptyset\\
1256&\emptyset & 1346&\emptyset & 1356&5\\
1456&5 & 2345&3 & 2346&4\\
2456&4 & 12345&\emptyset & 12346&\emptyset\\
12356&\emptyset & 12456&\emptyset & 13456&\emptyset\\
23456&3 & 123456&\emptyset &&
\end{array}
\]
It is not hard to check that for every row of the table, there
is no \(Q\in\cB\cup\{R_1,R_2\}\) such that \(|Q|\le |P|\) and
\begin{equation}\label{eq:six-certificate}
       Q\cap P=\sigma(P). 
\end{equation}
Let \(F=P\cup B\in\cF_6\), where \(P\in\cB\setminus\{1345,3456\}\). 

\begin{claim}
  \(B\cup\sigma(P)\) is missing from \(\Tr_{\cF_6}(F)\).  
\end{claim}
\begin{poc}
If an
edge from the \(\cB\)-part realizes this trace, then its core
\(Q\in\cB\) must satisfy \(|Q|\le |P|\) and \(Q\cap P=\sigma(P)\),
contradicting \eqref{eq:six-certificate}.  If an edge from the \(R_i\)-part
realizes the trace, then \(B\) must be contained in its outside part, so
\(|R_i|\le |P|\), and the core condition gives \(R_i\cap P=\sigma(P)\), again
contradicting \eqref{eq:six-certificate}. This proves the claim.
\end{poc}

We next handle the two exceptional patterns.  First take
\(F=1345\cup B\), where \(B\in\binom{U}{d-3}\). Notice that
there is no \(Q\in\cB\) such that
\begin{equation}\label{eq:six-special-1345}
     Q\cap 1345=35,
\end{equation}
and that \(R_1\cap 1345=145\), while \(R_2\cap 1345=35\).
If \(B\notin\cH_2\), then \(B\cup 35\) is missing.  The \(\cB\)-part is ruled
out by \eqref{eq:six-special-1345}; the \(R_1\)-recursive part has the wrong
core trace \(145\); and an \(R_2\)-recursive edge would realize \(B\cup35\)
only if its outside part equals \(B\), contrary to \(B\notin\cH_2\).

If \(B\in\cH_2\), then \(B\) is not shattered by \(\cH_2\), because
\(\VC(\cH_2)\le d-4\).  Since \(B\) itself is realized, choose
\(Y\subsetneq B\) such that
\(
        Y\notin \Tr_{\cH_2}(B).
\)
Then \(35\cup Y\) is missing from \(\Tr_{\cF_6}(F)\): the \(\cB\)-part is
ruled out by \eqref{eq:six-special-1345}, the \(R_1\)-part has the wrong core
trace, and the \(R_2\)-part would require some \(H\in\cH_2\) with
\(H\cap B=Y\), which does not exist.

Second take \(F=3456\cup B\), where \(B\in\binom{U}{d-3}\).  Here the
facts are that there is no \(Q\in\cB\) such that
\begin{equation}\label{eq:six-special-3456}
      Q\cap 3456=45,
\end{equation}
and that \(R_1\cap 3456=45\), while \(R_2\cap 3456=356\).
If \(B\notin\cH_1\), then \(B\cup45\) is missing.  If \(B\in\cH_1\), choose
\(Y\subsetneq B\) such that
\(
        Y\notin \Tr_{\cH_1}(B),
\)
which exists because \(\VC(\cH_1)\le d-4\).  Then \(45\cup Y\) is missing.
The verification is exactly parallel to the previous paragraph, using
\eqref{eq:six-special-3456} and the fact that only the \(R_1\)-recursive part
has core trace \(45\) on \(3456\).

It remains to consider recursive edges.  If \(F=R_1\cup H=1245\cup H\), where
\(H\in\cH_1\), then the trace \(4\) is missing.  Indeed, there is no
\(Q\in\cB\) such that
\(
        Q\cap R_1=4,
\)
the \(R_1\)-recursive edges all contain the whole core \(R_1\), and
\(R_2\cap R_1=25\ne 4\).
Thus no edge of \(\cF_6\) has trace \(4\) on \(F\).

Similarly, if \(F=R_2\cup H=2356\cup H\), where \(H\in\cH_2\), then the trace
\(3\) is missing.  To see this, there is no \(Q\in\cB\) such that
\(
        Q\cap R_2=3,
\)
the \(R_2\)-recursive edges all contain the whole core \(R_2\), and
\(R_1\cap R_2=25\ne 3\).
Therefore every edge of \(\cF_6\) has a missing trace.  By Fact~\ref{lem:missing-trace},
\(
        \VC(\cF_6)\le d.
\)
\end{proof}
Consequently,
\[
        M_d(n)\ge \binom{n-1}{d}+\binom{n-4}{d-2}+\binom{n-6}{d-3}+2M_{d-4}(n-6),
\]
which proves \cref{thm:main}.

\section{Some remarks}
Our constructions should not be viewed as the end of the story. The underlying idea is to use a carefully chosen 6-point core with two missing \(4\)-patterns that can
be filled recursively. We believe that if one replaces this 6-point core by a larger finite core, it might be possible
to further improve the coefficients of the
lower order error terms in our main result.  This is a finite problem in
principle, but it becomes complicated very quickly: a missing pattern that is
usable at one level might be destroyed by traces coming from higher levels, and
the certificate table has to be checked simultaneously with all recursive
insertions.

Thus, the present construction points to a sharp increase in the complexity of
the Erd\H{o}s--Frankl--Pach problem once \(d\ge3\). Accordingly, we should 
not expect the exact values in 
higher dimensions to 
be governed by the same rigidity phenomenon  as in the case \(d=2\).  The first natural test case is
\(d=3\).  
Part (i) of Theorem \ref{thm:main} gives
\[
        M_3(n)\ge \binom{n-1}{3}+n-3,
\]
for every \(n\ge 6\), and we suspect that this 
bound is sharp for all sufficiently large \(n\).
The restriction to sufficiently large \(n\) is necessary: it was shown in~\cite{2025ThreeUniform} that \(M_{3}(8)\ge 45,\) whereas \(\binom{7}{3}+5=40\).

\begin{conj}\label{conj:d=3}
For all sufficiently large \(n\),
\(
        M_3(n)=\binom{n-1}{3}+n-3.
\)
\end{conj}

\section*{Acknowledgement}
Zixiang Xu is grateful to Prof. Tuan Tran for inviting him to attending the Conference on Graph Theory at USTC in June 2026 and for his hospitality during the visit.
Before submitting the draft, the authors were informed by Prof. Gennian Ge and Xiaochen Zhao that they also independently obtained a new lower bound which improves the result of Ahlswede and Khachatrian~\cite{1997CombFan}.

\bibliographystyle{abbrv}
\bibliography{FranklPach}
\appendix


\section{The $4$-uniform construction}\label{appendix}
In this appendix, we prove the $4$-uniform case of Theorem \ref{thm:main}.




\begin{proof}[Proof of Theorem~\ref{thm:main}(i)]
We retain the notation \(C=[6]\), \(U=[n]\setminus C\), and \(\cB\) from the proof of Theorem \ref{thm:main}(ii). Define
$\cF_3=\big\{P\cup A:P\in\cB,\ A\in\binom{U}{4-|P|}\big\}$.
Then
\[
|\cF_3|=\sum_{s=1}^4|\cB_s|\binom{n-6}{4-s}
=\binom{n-1}{3}+\binom{n-4}{1}+1.
\]
It remains to show that
\(\VC(\cF_3)\le3\).
Suppose first that
\(P\notin\{1345,3456\}\), and let \(\sigma(P)\subsetneq P\) be the
certificate from the table in the proof of
Proposition~\ref{prop:VCD}. We claim that \(A\cup\sigma(P)\) is missing
from \(\Tr_{\cF_3}(F)\). Indeed, if \(Q\cup A'\in\cF_3\) realized this
trace, then \(A\subseteq A'\), and hence
\(4-|P|=|A|\le |A'|=4-|Q|\), so \(|Q|\le|P|\). Comparing the
intersections with \(C\) gives \(Q\cap P=\sigma(P)\), contradicting
the certificate property.

If \(P=1345\), then \(A=\emptyset\), and the trace \(35\) is missing
because no \(Q\in\cB\) satisfies \(Q\cap1345=35\). Similarly, if
\(P=3456\), then the trace \(45\) is missing. Thus every member of
\(\cF_3\) has a missing trace, and Fact~\ref{lem:missing-trace} yields
\(\VC(\cF_3)\le3\).
\end{proof}

\end{document}